\documentclass[12pt]{amsart}

\newtheorem{theorem}{Theorem}[section]

\newtheorem{corollary}[theorem]{Corollary}

\newtheorem{lemma}[theorem]{Lemma}
\newtheorem{proposition}[theorem]{Proposition}

\theoremstyle{definition}

\newtheorem{definition}[theorem]{Definition}

\newtheorem{example}[theorem]{Example}

\theoremstyle{parrafo}

\begin{document}

\title[]{Generalized Bernstein operators defined by increasing nodes}

\author{J. M. Aldaz and H. Render}
\address{Instituto de Ciencias Matem\'aticas (CSIC-UAM-UC3M-UCM) and Departamento de 
	Matem\'aticas,
	Universidad  Aut\'onoma de Madrid, Cantoblanco 28049, Madrid, Spain.}
\email{jesus.munarriz@uam.es}
\email{jesus.munarriz@icmat.es}
\address{H. Render: School of Mathematical Sciences, University College
	Dublin, Dublin 4, Ireland.}
\email{hermann.render@ucd.ie}


\thanks{The first named author was partially supported by Grant MTM2015-65792-P of the
	MINECO of Spain, and also by by ICMAT Severo Ochoa project SEV-2015-0554 (MINECO)}

\begin{abstract} 
We study  certain generalizations of the classical Bernstein operators, 
defined via increasing sequences of nodes. Such operators are required to
fix 
two functions, $f_0$ and $f_1$, such that $f_0 > 0$ and $f_1/ f_0$ is increasing
on an interval $[a,b]$.
 A characterization regarding  when this can be done is presented. From it we
 obtain, under rather general circumstances, the following necessary  condition
for existence:  if nodes are non-«decreasing, then $(f_1/f_0)^\prime >0 $ on $(a,b)$, 
while if nodes are strictly increasing,  then $(f_1/f_0)^\prime >0 $ on $[a,b]$.
\end{abstract}

\maketitle

\markboth{J. M. Aldaz, H. Render}{Generalized Bernstein Operators}

\section{Introduction}

Let $\mathbb{P}_{n} = \mathbb{P}_{n}[a,b]$ denote the space of polynomials of degree
bounded by $n$, over the interval $[a,b]$. In recent years there has been a continued interest in finding generalizations or modifications of the 
classical Bernstein operators 
$B_{n}:C
\left[ a,b\right] \rightarrow \mathbb{P}_{n}[a,b]$, defined by 
\begin{equation} 
B_{n}f\left( x\right)
=
\sum_{k=0}^{n}f\left( a+\frac{k}{n}\left( b-a\right)
\right) \binom{n}{k}\frac{\left( x-a\right) ^{k}\left(
	b-x\right) ^{n-k}}{\left( b-a\right) ^{n}},  \label{defBPr}
\end{equation}
to more general spaces of functions, but still reproducing a two dimensional
subspace, say $\operatorname{Span}\{f_0, f_1\}$, with
$f_0 > 0$ and $f_1/f_0$ injective. Within the realm of
polynomial spaces, one asks when the exact reproduction of functions, other than the affine ones, is possible.
Also, similar questions have been asked about related positive
operators, (cf. for instance \cite{AcArGo}).

Sometimes fixing a subspace different from the affine functions
 is achieved by modifying the Bernstein
bases
(consider, for instance, the nowadays called King's operators,
after \cite{Ki}). 
Within the line of research followed here  (cf. \cite{MoNe00}, \cite{AKR07},
\cite{AKR08b}, \cite{AKR08}, \cite{KR07b}, \cite{AR}, \cite{Ma}, \cite{AiMa},
\cite{AlRe18})
fixing $f_0$ and   $f_1$ is achieved,
when possible, by modifying the location of the nodes $t_{n, k }$ 
(instead of having  $t_{n, k } =  a+\frac{k}{n}\left( b-a\right)$ as in (\ref{defBPr})). A motivation for this approach
is that it allows us to keep the Bernstein bases unchanged, a desirable feature given
their  several  optimality properties, cf. for instance \cite{Fa}. Multiplying by $-1$ if needed,
we may assume that $f_1/f_0$ is strictly increasing. 

The situation regarding the existence of generalized Bernstein operators,
defined by strictly increasing sequences of nodes,  is well understood in the context of
extended Chebyshev spaces, cf. \cite{AKR08b}: one considers a two
dimensional extended Chebyshev space $U_1$,  for which a generalized
Bernstein operator fixing it can always be defined with increasing
nodes (since they are the endpoints of the interval), 
and inductively, via the interlacing property of nodes (cf. \cite[Theorem 6]{AKR08b}) 
this  definition is extended to
$U_1 \subset U_2 \subset \cdots \subset U_n$, where  each $U_k$ is a
$k + 1$-dimensional extended Chebyshev space.
 
 But this framework is insufficient even for spaces of polynomials,
 since for instance, it cannot handle the case where we have
 $U_1: = \operatorname{Span}\{\mathbf{1}, x^3\}$ over $[a,b]$, with
$a < 0 < b$ (cf. Example \ref{E2}). It is thus natural to try to go beyond
chains of extended Chebyshev spaces. Now, a salient difference between
various definitions of generalized Bernstein operators appearing in the literature, 
 is whether one should
require the sequence of nodes to be strictly increasing
(as in \cite{Ma}), 
or not (as in \cite{AKR08b}, \cite{AKR08}). Of course, having 
strictly increasing nodes leads to better properties from the point of view of
shape preservation,  but existence will be obtained   in fewer cases. 

We shed light on this issue by characterizing, in terms of the
spaces $U_n$ and $D_{f_{0}}U_{n}:=\left\{ \frac{d}{dx}\left( \frac{f}{f_{0}}\right) :f\in
U_{n}\right\}$, when the sequence of nodes is non-decreasing, and
when it is strictly increasing, cf. 
 Theorem \ref{ThmBern2} below for full details.  
 This Theorem  improves on 
 \cite[Theorem1]{AKR07} (the main result of  \cite{AKR07}) and generalises 
 \cite[Theorem 3.2]{AlRe18}, which deals exclusively with the polynomial case
 $\mathbb{P}_{n}[a,b]$. From Theorem \ref{ThmBern2} 
 the following necessary condition is obtained: 
 if both spaces $U_n$ and $D_{f_{0}}U_{n}:=\left\{ \frac{d}{dx}\left( \frac{f}{f_{0}}\right) :f\in
 U_{n}\right\}$ have positive Bernstein bases (a hypothesis weaker than
 being extended Chebyshev spaces) then the existence of a generalized 
 Bernstein operator having non-decreasing nodes entails 
 that $(f_1/f_0)^\prime >0 $ on $(a,b)$, 
 while if nodes are strictly increasing,  then $(f_1/f_0)^\prime >0 $ on $[a,b]$,  cf. Corollary \ref{antimaz}.
 
  Since this result contradicts some statements
 made in \cite[Section 7.2]{Ma}, in an effort to  clarify these issues
 we have emphasized concrete examples
 and explicit computations thoroughout the paper. 
 
To sum up, the difference between the cases where $f_1^\prime$ vanishes
at some point inside $(a,b)$, and where $f_1^\prime >0$
on $(a,b)$, turns out to be very important
from the point of view of the ordering of the nodes, and hence, of shape preservation
and of the existence of generalized Bernstein operators. 

\section{Definitions and motivating examples.}

\begin{definition} Let $U_{n}$  be an $n+1$ dimensional subspace of 
	$C^n\left( \left[
	a,b\right], \mathbb{K}\right)$, where $\mathbb{K} = \mathbb{R}$ or
	$\mathbb{K} = \mathbb{C}$.  A 
	Bernstein basis $\{p_{n,k}: k=0,\dots,n\}$ of $U_n$ is a basis
	with the property that each $p_{n,k}$ has a zero of order $k$ at $a$, and
	a zero of order 
	$n-k$ at $b$. The function $p_{n,k}$ might have additional zeros inside 
	$\left(a,b\right) $; this is not excluded by the preceding definition. A Bernstein basis 
 is 
 {\em non-negative} if for all
	$k= 0, \dots, n$, $p_{n,k} \ge 0$ on  $\left[ a,b\right]$, and
	{\em positive} if $p_{n,k} > 0$ on $\left(a,b\right)$. Finally,
	a non-negative Bernstein basis is {\em normalized} if $\sum_{k=0}^n p_{n,k} \equiv 1$.
\end{definition} 

It is easy to check that non-negative Bernstein bases are unique
up to multiplication by a positive scalar, and that normalized
Bernstein bases are unique.

\begin{definition} If $U_{n}$ has a non-negative Bernstein basis $\{p_{n,k}: k=0,\dots,n\}$, we define a
	{\em generalized Bernstein operator} $B_{n}:C\left[ a,b\right] \rightarrow U_{n}$
	by setting 
	\begin{equation}
	B_{n}\left( f\right) =\sum_{k=0}^{n}f\left( t_{n,k}\right) \alpha
	_{n,k}p_{n,k},  \label{eqBern}
	\end{equation}
	where the nodes $t_{n,0},...,t_{n,n}$ belong to the interval 
	$\left[ a,b\right]$, and the weights $\alpha_{n,0},...,\alpha_{n,n}$ are positive. 
\end{definition} 

Non-negativity of the functions $p_{n,k}$ and positivity of the weights $\alpha_{n,0},...,\alpha_{n,n}$ are required
so that the resulting operator is  positive, a natural property from the viewpoint of
shape preservation. Strict positivity
of the weights entails that all the basis functions are used in
the definition of the operator, something useful if we want families of
operators to converge to the identity. Finally,
 the nodes must belong to $\left[ a,b\right]$; otherwise, the operator will not be well defined on
$C\left[ a,b\right]$. But no requirement is made  about the
ordering of the nodes, and in particular, we do not ask that they  be {\em strictly  increasing},
i.e., 
that  $t_{n,0} < t_{n,1} < \cdots < t_{n,n}$. When we only have
$t_{n,0} \le t_{n,1} \le \cdots \le t_{n,n}$ we say that the sequence of nodes is
{\em increasing}, or equivalently, {\em non-decreasing}.

\vskip .2 cm

The problem of
existence, as studied in  \cite{AKR08b} and \cite{AKR08}, arises when we choose two functions $f_{0},f_{1}\in U_{n}$, such that $f_{0} >0 $,  
$f_{1}/f_{0}$ is strictly increasing, and we  require that 
\begin{equation}
B_{n}\left( f_{0}\right) =f_{0}\text{ and }B_{n}\left( f_{1}\right) =f_{1}.
\label{eqBern2}
\end{equation}
If these equalities can be satisfied, they uniquely determine the location of the
nodes and the values of the coefficients, cf. \cite[Lemma 5]{AKR08b}; in other words, there is at most
one Bernstein operator $B_{n}$ of the form (\ref{eqBern}) satisfying (\ref
{eqBern2}).  
We will consistently use the following notation. Assume that $p_{n,k},$ 
$k=0,...,n$, is a Bernstein basis of the space $U_{n}$. Given 
$f_{0},f_{1}\in U_{n}$, there exist coefficients $\beta_{n,0},...,\beta_{n,n}$
and $\gamma_{n,0},...,\gamma_{n,n}$ such that 
\begin{equation}
f_{0}\left( x\right) =\sum_{k=0}^{n}\beta_{n,k}p_{n,k}\left( x\right) \text{
and }f_{1}\left( x\right) =\sum_{k=0}^{n}\gamma_{n,k}p_{n,k}\left( x\right) .
\label{eqeq}
\end{equation}

The following elementary fact regarding bases will be used throughout
(cf. \cite[Lemma5]{AKR08b}):
If there exists a generalized Bernstein operator $B_n$ of the form
 (\ref{eqBern}), fixing $f_0$ and $f_1$,
then it must be the case that for each $k= 0, \dots, n$, 
\begin{equation}
\beta_{n,k} = f_0 (t_{n,k}) \ \alpha_{n,k} \text{ \ \ \ 
	and \ \ \ } \gamma_{n,k}  = f_1 (t_{n,k}) \  \alpha_{n,k}.
\label{bases}
\end{equation} 
Note that since  by hypothesis $f_0 > 0$ and  $\alpha_{n,k} > 0$,
if $B_n$ exists then  $\beta_{n,k} > 0$.
Now, using the injectivity of $f_1/f_0$,
the
nodes are uniquely determined by
\begin{equation}
t_{n,k}:=\left( \frac{f_{1}}{f_{0}}\right) ^{-1}\left( \frac{\gamma_{n,k}}{%
	\beta_{n,k}}\right),  \label{nodes}
\end{equation}
and the weights, by
\begin{equation}
\alpha_{n,k}:=\frac{\beta_{n,k}}{f_{0}(t_{n,k})}. \label{coeff}
\end{equation}
Finally, when the Bernstein basis is non-negative and normalized,
and $f_0 = \mathbf{1}$, we have that
\begin{equation} \label{one}
1 = \alpha_{n,k} = \beta_{n,k}
 \text{ \ \ \ 
	and \ \ \ } t_{n,k}= f_{1}^{-1}\left(\gamma_{n,k}\right).
\end{equation}

Suppose that instead of fixing $ \mathbf{1} $ and $x$ over $[a,b]$, we want a
generalized Bernstein operator that reproduces $f_0 =  \mathbf{1} $ and some strictly
increasing function other than $x$. Possibly the simplest choice is
to fix $f_1 (x) = x^3$, since $x^2$ is not increasing over arbitrary intervals.
Already in this case we observe a wide range of behavior, depending
on the values of $a$ and $b$. 

Recall that  $\mathbb{P}_n [a,b]$ denotes the space of polynomials 
on $[a,b]$, of degree bounded
by $n$. In this case the  Bernstein bases are given by
$p_{n,k}  (x) = \binom{n}{k}
\frac{(x - a)^{k}(b-x)^{n-k}}{(b - a)^n}$.

\begin{example}\label{E1} Consider  $\mathbb{E}_1 = 
\operatorname{Span}\{\mathbf{1},x^3\}$ over
$[-1,1]$. Then $\{p_{1,0}(x) := (1-x^3)/2, p_{1,1}(x) := (1 + x^3)/2\}$ is the unique normalized
Bernstein basis for $\mathbb{E}_1$.
Define, as in \cite[Formula (21)]{AKR08b}, 
\begin{equation} \label{dim1}
B_{1}f:= f\left(  a\right)  p_{1,0}+f\left(  b\right) p_{1,1}.
\end{equation}
Then it is clear that  $B_1 \mathbf{1} =  p_{1,0}+p_{1,1} = \mathbf{1}$ and 
$B_1 f_1 = - p_{1,0}+p_{1,1} =x^3$. Note that 
 $\operatorname{Span}\{\mathbf{1},x^3\}$ is not
an extended Chebyshev space, a notion defined  next.
\end{example}

\begin{definition} \label{ECS} An \emph{extended Chebyshev
space}  $U_{n}$ of dimension $n+1$ over the interval $\left[ a,b\right] $ 
 is an $n+1$ dimensional subspace of $C^{n}\left( \left[
a,b\right] \right) $ such that each $f\in U_{n}$ has at most $n$ zeros in 
$\left[ a,b\right] $, counting multiplicities, unless $f$ vanishes
identically. 
\end{definition}

Extended Chebyshev
spaces  of dimension $n+1$  generalize
the space of polynomials of degree at most $n$ by retaining the bound on the
number of zeros.
It is well known that extended Chebyshev spaces always have
positive Bernstein bases.

\begin{example}\label{E2} Let $\mathbb{E}_2 =\operatorname{Span}\{\mathbf{1},x, x^3\}$ over
$[a,b] = [-1,1]$. In this case it is impossible to obtain a non-negative Bernstein basis
for $\mathbb{E}_2 $, whence the corresponding generalized Bernstein operator
cannot be defined. To see why this is true,  note that if $p_{2,1} (x)  = a + b x + c x^3$
has one zero at $-1$ and another at $1$, then 
$a - b -c = a + b + c = 0$, so $a = 0$ and $b = -c$, with $b \ne 0$. But for any such
$b$, $p_{2,1} (x)  = b (x - x^3)$ crosses the $y$-axis at $0$. 
Note that Bernstein bases do exist for $\mathbb{E}_2$:
One such
 basis  is given by
$\{p_{2,0} (x) = 2 - 3 x +   x^3,  p_{2,1} (x) =  x -   x^3,  p_{2,2} (x) = 2  + 3 x -  x^3\}$.

Let us now consider $\mathbb{E}_2  = \operatorname{Span}\{\mathbf{1},x, x^3\}$ over
	$[-1,2]$. In this case it is impossible to obtain a Bernstein basis
	for $\mathbb{E}_2 $, even allowing for changes of signs. Suppose
	there is such a basis, and let us try to compute $p_{2,2}$. Note that the
	coefficient of $x^3$ cannot be zero, since $ p_{2,2} (x) $ has degree at least
	two (hence three);  dividing by the said coefficient, we may assume that
	$p_{2,2} (x) = a  + b x  +  x^3 = ( x + 1)^2 ( x + c) = c + (1 + 2c) x + (2 + c) x^2 + x^3$.
Equating coefficients we see that $c = - 2$. Hence $p_{2,2} ( 2)  = 0$, which is a contradiction.
\end{example}

\begin{example}\label{P3} Next we consider 
	$\mathbb{P}_3[a,b] =\operatorname{Span}\{\mathbf{1},x, x^2, x^3\}$,  with the standard
	Bernstein bases over 
	$[a, b] = [-1,1]$ and over 	$[a, b] = [-1,2]$. In 
	the first case a generalized Bernstein operator fixing 
	$f_0 = \mathbf{1}$ and 
	$ f_1 (x)  = x^3$ exists, but the
	sequence of nodes fails to be increasing. In fact, this must be
	the case, by Corollary \ref{antimaz} below, since 
	$ f_1^\prime (0) = 0$. 
	More explicitly, solving for
	the coefficients  $\gamma_{3,k}$ of $x^3$ we have
	$\gamma_{3,0} = -1$, $\gamma_{3,1}= 1$, $\gamma_{3,2} = -1$, and 
	$\gamma_{3,3} = 1$, so in this particular instance the coordinates and the nodes take the
	same values, oscillating between $-1$ and $1$. Note that 
	$B_3$ is just the projection from $C\left[ a,b\right]$ onto 
	$\operatorname{Span}\{\mathbf{1}, x^3\}$. This can be seen by observing
	that for $k \ge 0$, $B_3 x^{2 k} = 1$ and  $B_3 x^{2 k + 1} = x^3$.
	Alternatively, given $f \in C\left[ a,b\right]$, if we simplify the
	expression for $B_3 f(x)$,  we find that 
	$B_3 f(x) \in \operatorname{Span}\{\mathbf{1}, x^3\}$.

	 When  $[a,b] =[-1,2]$, a generalized Bernstein operator fixing
	$\mathbf{1}$ and $x^3$ does not exist on $\mathbb{P}_3 [a,b]$, since solving for
	the coefficients  $\gamma_{3,k}$ of $x^3$ we find that
	$\gamma_{3,0} = -1$, $\gamma_{3,1}= 2$, $\gamma_{3,2} = -4$, $\gamma_{3,3} = 8$,
	so the  node    
	$t_{3,2} = (-4)^{1/3}$ falls outside $[-1,2]$.

	However, 
	a generalized Bernstein operator fixing
	$\mathbf{1}$ and $x^3$ does  exist on $\mathbb{P}_4 [-1,2]$, for now the coefficients  $\gamma_{4,k}$ of $x^3$ are
	$\gamma_{4,0} = -1$, $\gamma_{4,1}= 5/4 $, 
	$\gamma_{4,2} = -1$, $\gamma_{4,3} = - 1$
	and $\gamma_{4,4}= 8$, so all the nodes fall inside $[-1 , 2]$.

	Thus, not only the cases where $f_1$ is strictly increasing and where 
	$f_1^\prime > 0$ on $(a,b)$ are different, but also the (relative) location of the possible
	zeros of $f_1^\prime$ is relevant; a more extreme instance of this phenomenon 
	can be found in \cite[Theorem 5.2]{AlRe18} 
	\end{example}

Next we present some counterexamples. The results are analogous to the
instances  seen so far, but spaces and bases are chosen to specifically address
some claims made in \cite[pages 121-122 ]{Ma}, where it is stated ``Our purpose here is not to develop a comprehensive theory on Bernstein-type operators, but to convince the 
reader  via a few relevant examples, that there do exist similar operators in more general situations."  While  Bernstein-type operators can be defined in more general situations,
they do not exist in some of the relevant examples presented there.
The following spaces are considered in \cite[page 123]{Ma}: 
Let $a < 0 < b$, and 
let $n\ge 4$. Consider the
sequence 
$
\mathbb{E}_1\subset \mathbb{E}_2\subset \mathbb{P}_{3} \subset \cdots\subset \mathbb{P}_{n-1}\subset \mathbb{E}_n,
$
 where $\mathbb{E}_1 := \operatorname{Span}\{ \mathbf{1}, x^3\}$, $\mathbb{E}_2 := 
\operatorname{Span}\{  \mathbf{1}, x, x^3\}$, and for $n \ge 4$,  $\mathbb{E}_n := 
\operatorname{Span}\{  \mathbf{1}, x,  \dots, x^{n-1},  x^{n+ 2}\}$. 
The domain of definition of these functions is taken to be the interval $[a,b]$. 
In \cite[Definitions 3.1 and 3.2]{Ma} ``Bernstein-like operators" are required
to have strictly increasing sequences of nodes.

Now in  \cite[Example 7.1]{Ma},  the
existence of ``Bernstein-like operators" 
$B_{n}:C\left[ a,b\right] \rightarrow \mathbb{E}_n$
 for  $a < 0 < b$,  fixing
 $\mathbf{1}$ and $x^3$, is asserted. We show here, by explicit computation, that 
when $n= 4$ and $[a,b]= [-1,2]$, a generalized Bernstein operator
fixing $\mathbf{1}$ and $x^3$ 
does not exist. 
When $[a,b]= [-1,1]$, such an   operator exists, but  one must give up
the condition of increasing nodes.

\begin{example} \label{ex1} First we take $[a,b] =[-1,1]$. In this case a generalized Bernstein operator $B_{4}:C\left[ a,b\right] \rightarrow \mathbb{E}_4$ fixing
$ \mathbf{1}$ and $x^3$ does exist, but it is not defined
via an increasing sequence of nodes.
The normalized Bernstein basis on $\mathbb{E}_4$ can simply be found
by writing  arbitrary linear combinations of the functions
$\{  \mathbf{1}, x, x^2, x^3, x^6\}$, and imposing the conditions of having precisely 4 zeros,
 $k$  of them at $-1$ and the other
$4-k$ at $1$. Multiplying by $-1$ if needed, these basis functions can
be assumed to be non-negative at 0 (in fact, we shall see that they  are positive
inside $(-1, 1)$).  Imposing the additional condition that they
add up to 1, we find the
unique normalized Bernstein basis $\{p_{4,0}, \dots,p_{4,4}\}$
on $\left[ -1,1\right] $, where
\begin{eqnarray*}
p_{4,0}\left( x\right) &=&\frac{5}{56} - \frac{9}{28} x + \frac{45}{112}  x^2  - \frac{5}{28} 
    x^3  + \frac{1}{112} x^6, \\
p_{4,1}\left( x\right) &=&\frac{2}{7} - \frac{3}{7} x - \frac{3}{14}  x^2  + \frac{3}{7} 
    x^3  - \frac{1}{14} x^6, \\
    p_{4,2}\left( x\right) &=&\frac{1}{4} -  \frac{3}{8}  x^2  + \frac{1}{8} x^6, \\
   p_{4,3}\left( x\right) &=&\frac{2}{7} + \frac{3}{7} x - \frac{3}{14}  x^2  - \frac{3}{7} 
    x^3  - \frac{1}{14} x^6,\\
p_{4,4}\left( x\right) &=&\frac{5}{56} + \frac{9}{28} x + \frac{45}{112}  x^2  + \frac{5}{28} 
    x^3  + \frac{1}{112} x^6.
\end{eqnarray*}
These functions are positive at zero, add up to 1, and have the 
correct number of
 zeros at the endpoints. To see
that they form a positive basis, since $p_{4,k}(0) > 0$ for
$k=0,\dots, 4$, 
it suffices to show that they have no additional zeros inside
$(-1,1)$. But this is easily checked, for we already know
the location of four zeros of each $p_{4,k}$. Using the division
algorithm, we factor all the corresponding linear terms $(x +1)$ and
$(x - 1)$, and are left
in each case with a second degree polynomial having no real roots. 

Once we have found the normalized Bernstein bases, we use the condition
(\ref{eqBern2}) to determine nodes:
 Equating coefficients in 
$
 x^3 =\sum_{k=0}^{4}\gamma_{4,k}p_{4,k}\left( x\right)
 $
and solving for $\gamma_{4,k}$ we find that
$\gamma_{4,0} = -1$, $\gamma_{4,1}= 3/4$, $\gamma_{4,2} = 0$, $\gamma_{4,3} = -3/4$
and $\gamma_{4,4}= 1$.
Since the nodes are the cube roots of these coordinates, it follows
that $t_{4,0} < 
t_{4,3} < t_{4,2} < t_{4,1} < t_{4,4}$, and we see that the nodes do not form an increasing sequence in $k$.
\end{example}

\begin{example} \label{ex2} Let us now take $[a,b] =[-1,2]$. In this case a generalized Bernstein operator  $B_{4}:C\left[ a,b\right] \rightarrow \mathbb{E}_4$ fixing
$\mathbf{1}$ and $x^3$, cannot be defined.
Using the same steps as in the preceding example, we find  the following
 Bernstein basis functions:
\begin{eqnarray*}
p_{4,0}\left( x\right) &=& \frac{640}{2673} - \frac{128}{297} x + \frac{80}{297}  x^2  - \frac{160}{2673} 
    x^3  + \frac{1}{2673} x^6, 
    \\
p_{4,1}\left( x\right) &=& \frac{5776}{13365} - \frac{152}{1485} x - \frac{532}{1485}  x^2  + \frac{2318}{13365} 
    x^3  - \frac{38}{13365} x^6, 
     \\
    p_{4,2}\left( x\right) &=& \frac{98}{405} + \frac{14}{45} x - \frac{7}{90}  x^2  - \frac{56}{405} 
    x^3  + \frac{7}{810} x^6,
 \\
   p_{4,3}\left( x\right) &=& \frac{16}{243} + \frac{4}{27} x + \frac{2}{27}  x^2  - \frac{4}{243} 
   x^3  - \frac{2}{243} x^6,
   \\
p_{4,4}\left( x \right) &=& \frac{5}{243} + \frac{2}{27} x + \frac{5}{54}  x^2  +  
\frac{10}{243} x^3  + \frac{1}{486} x^6.
\end{eqnarray*}
Positivity of these functions on $(-1,2)$ is obtained by noticing, first, that
$p_{4,k} (0) = 1$, and second, that after factoring the  linear terms $(x +1)$ and
$(x - 2)$, 
in each case we are left  with a second degree polynomial having no real roots. 
Adding up we see that the basis is normalized,
so it is enough to compute the coordinates $\gamma_{4, k}$ of $x^3$. Doing so,
we find that $\gamma_{4, 2} = - 16/7 < -1$, 
so the node $t_{4, 2} \notin [-1, 2]$. 
\end{example}

\section{Characterizing when nodes  increase for general spaces.}

The following technical results, 
used to prove
Theorem \ref{ThmBern2}., come from \cite{AKR08}.     Proposition \ref{PropABL} appears in 
 \cite[Proposition 3]{AKR08}, while Lemma \ref{LemA} is a less
 general version of  \cite[Lemma 6]{AKR08}. 
 
\begin{proposition}
\label{PropABL} Assume that $U_{n}\subset C^n([a,b], \mathbb{K})$ has a Bernstein basis $p_{n,k},k=0,...,n$.
Let $f_{0}\in U_{n}$ be strictly positive and suppose that $D_{f_{0}}U_{n}:=\left\{ \frac{d}{dx}\left( \frac{f}{f_{0}}\right) :f\in
U_{n}\right\}$ 
has a Bernstein basis $q_{n-1,k}$, $k=0,...,n-1$. Set $c_0 := 0$, $q_{n-1, -
1} := 0$, $d_n:= 0$, and $q_{n-1, n} := 0$. For $k=1,...,n$, define the
non-zero numbers
\begin{equation}
c_{k}:= \frac{1}{f_{0}\left( a\right) }\lim_{x\downarrow a}\frac{\frac{d}{dx}%
p_{n,k}\left( x\right) }{q_{n-1,k-1}\left(x\right)} = \frac{1}{f_{0}\left(
a\right) }\frac{p_{n,k}^{\left( k\right) }\left( a\right) }{%
q_{n-1,k-1}^{\left( k-1\right) }\left( a\right) }  \label{eqPR1}
\end{equation}
and for $k=0,...,n-1$, the non-zero numbers
\begin{equation}
d_{k}:=\frac{1}{f_{0}\left( b\right) }\lim_{x\uparrow b}\frac{\frac{d}{dx}%
p_{n,k}\left( x\right) }{q_{n-1,k}\left(x\right)} = \frac{1}{%
f_{0}\left(b\right) }\frac{p_{n,k}^{\left( n-k\right) }\left( b\right) }{%
q_{n-1,k}^{\left( n-1-k\right) }\left( b\right) }.  \label{eqPR}
\end{equation}
Then for every $k=0,...,n$,
\begin{equation}
\frac{d}{dx}\frac{p_{n,k}\left( x\right) }{f_{0}\left( x\right) }
=c_{k}q_{n-1,k-1}\left( x\right) +d_{k}q_{n-1,k}\left( x\right).
\label{eqPREC}
\end{equation}
\end{proposition}

\begin{lemma}
\label{LemA} Let $p_{n,k}$, $k=0,...,n$, be a  non-negative Bernstein
basis of $U_{n}\subset C^n([a,b], \mathbb{K})$. Then there exists a $\delta >0$ such that 
$
p_{n,k}^{\prime }\left( x\right) <0$ for all $x\in \left[ b-\delta ,b\right]
$ and all $k=0,...,n-1$, while $p_{n,k}^{\prime }\left( x\right) > 0$ for
all $x\in \left[a, a + \delta \right]$ and all $k=1,...,n$. Thus, the
numbers $c_{k}$ defined in (\ref{eqPR1}) for $k = 1, \dots, n$ are positive,
and the numbers $d_{k}$ defined in (\ref{eqPR}) for $k = 0, \dots, n-1$ are
negative.
\end{lemma}

\begin{theorem}
\label{ThmBern2} Assume that both $U_{n}\subset C^n([a,b], \mathbb{K})$ and 
$D_{f_{0}}U_{n}:=\left\{ \frac{d}{dx}\left( \frac{f}{f_{0}}\right) :f\in
U_{n}\right\}$ possess 
non-negative Bernstein basis $p_{n,k}$,  for $k=0, \dots , n$,   and $q_{n-1,k}$,
for $k=0, \dots , n-1$, respectively. Suppose $f_{0},f_{1}\in U_{n}$ are such that $f_{0}>0$,
its  coordinates $\beta_{n, k}$  satisfy
$\beta_{n, k} > 0$, and
$f_{1}/f_{0}$ is strictly increasing on $\left[  a,b\right]  $. Then the
following statements are equivalent:

a) There exists a generalized Bernstein operator $B_{n}:C\left[ a,b\right]
\rightarrow  U_n$ defined by a sequence of non-decreasing
(resp. strictly  increasing)  nodes,  and
 fixing both $f_{0}$ and $f_{1}.$

b) For $k=0, \dots , n$,   the numbers $\frac{\gamma_{n,k}}{\beta_{n, k}}$ are non-decreasing
(resp. strictly increasing).

c) The coefficients $w_{k}$, defined by
\begin{equation}
\frac{d}{dx}\frac{f_{1}}{f_{0}}=\sum_{k=0}^{n-1}w_{k}q_{n-1,k}
 \label{eqA}
\end{equation}
for $k=0, \dots , n-1$,
are non-negative (resp. strictly positive).

\end{theorem}

\begin{proof}
We begin  with some technical preliminaries, under
the assumption that for $k=0, \dots , n$, the coordinates $\beta_{n, k}$ of $f_0$   
are strictly positive.  Let $k_{0}\in\left\{
0, \dots ,n-1\right\}  .$ Since $p_{n,k},k=0,...,n$, is a basis, there exists
numbers $\delta_{1},...,\delta_{n}$ such that
\begin{equation}
\psi_{k_{0}}:=f_{1}-\frac{\gamma_{n,k_0}}{\beta_{n,k_0}}f_{0}=\sum_{k=0}
^{n}\delta_{k}p_{n,k}. 
\label{eqfneu}
\end{equation}
From (\ref{eqeq}) we get
\begin{equation*}
\delta_{k}=\gamma_{n,k}-\frac{\gamma_{n,k_0}}{\beta_{n,k_0}}\beta_{n, k}
\end{equation*}
for $k=0,...,n.$ Setting  $k=k_{0}$ we obtain
\begin{equation}
\label{berob}
\delta_{k_{0}}=\gamma_{n,k_0}-\frac{\gamma_{n,k_0}}{\beta_{n,k_0}}
\beta_{n, k_{0}}=0.
\end{equation}
Let us write
 $$
 \frac{f_{1}}{f_{0}}-\frac{\gamma_{n,k_0}}{\beta_{n,k_0}}
=
 \frac{\psi_{k_{0}}}{f_{0}}
=
\sum_{k=1}^{n}\delta_{k}\frac{p_{n,k}}{f_{0}}.
$$
Differentiating  we get
\[
\frac{d}{dx}\frac{f_{1}}{f_{0}}=\sum_{k=0}^{n}\delta_{k}\frac{d}{dx}\left(
\frac{p_{n,k}}{f_{0}}\right)  .
\]
Proposition \ref{PropABL} together with Lemma \ref{LemA} show that
\[
\frac{d}{dx}\frac{f_{1}}{f_{0}}=\sum_{k=0}^{n}\delta_{k}\left[
c_{k}q_{n-1,k-1}+d_{k}q_{n-1,k}\right],
\]
where $c_0 = 0$ and $c_k > 0$ for $k=0, \dots , n$, while
$d_k < 0$ for $k=0, \dots , n-1$  and $d_n = 0$.
Thus,
\[
\frac{d}{dx}\frac{f_{1}}{f_{0}}=   
\delta_{0} d_{0}q_{n-1,0}
+
\sum_{k=1}^{n-1}\delta_{k}\left[
c_{k}q_{n-1,k-1}+d_{k}q_{n-1,k}\right]  +c_{n}\delta_{n}q_{n-1,n-1}
\]
\[
=   
\sum_{k=1}^{n}\delta_{k} 
c_{k} q_{n-1,k-1}
+
\sum_{k=0}^{n-1}\delta_{k} d_{k}q_{n-1,k} 
=
\sum_{k=0}^{n-1}\left(  \delta_{k+1}c_{k+1}+\delta_{k}d_{k}\right)  q_{n-1,k}.
\]
Using (\ref{eqA}) we conclude that 
\begin{equation}
c_{k+1}\delta_{k+1}=w_{k}-\delta_{k}d_{k} 
\label{eqck}
\end{equation}
for $k= 0 ,\dots, n-1$. Inserting $k=k_{0}$ in (\ref{eqck}), from (\ref{berob}) we get
\begin{equation}
c_{k_{0}+1}\delta_{k_{0}+1}=w_{k_{0}}-\delta_{k_{0}}d_{k_{0}}=w_{k_{0}}
\label{eqnew11}
\end{equation}
whenever $k_{0}\in\left\{
0, \dots ,n-1\right\}  .$
Now the result is easily obtained from this equality.  We mention only the non-decreasing case,
since the strictly increasing one is handled in an identical manner.

 First we prove that a) and b) are 
equivalent. It follows from (\ref{eqeq}) that 
$\frac{f_1 (a) }{f_0 (a) } = \frac{\gamma_{n,0}}{\beta_{n, 0}}$
and $\frac{f_1 (b) }{f_0 (b) } = \frac{\gamma_{n,n}}{\beta_{n, n}}$.
 By (\ref{nodes}),  the nodes are non-decreasing (and hence they belong to $[a,b]$)  
 if and only if  so
are the numbers $\frac{\gamma_{n,k}}{\beta_{n, k}}$.

Regarding the equivalence between b) and c),  by (\ref{eqnew11}) we have 
$w_{k_{0}}\geq 0$ if and only if $\delta_{k_{0}+1}=\gamma_{k_{0}+1}-\frac{\gamma
_{k_{0}}}{\beta_{n,k_0}}\beta_{k_{0}+1}\geq0$, which is clearly equivalent to 
$\frac{\gamma_{n,k_0}}{\beta_{n,k_0}}\leq\frac{\gamma_{k_{0}+1}}{\beta
_{k_{0}+1}}$.
\end{proof}

\begin{corollary}\label{antimaz}
Let $f_{0}>0$,  let $f_{1}/f_{0}$ be strictly increasing on $\left[  a,b\right]
$, and let   $f_{0},f_{1}\in U_{n}\subset C^n([a,b], \mathbb{K})$. Suppose that $U_{n}$ 
has a non-negative Bernstein basis, and that $D_{f_{0}}U_{n}$ possesses a positive Bernstein basis. If there exists a generalized Bernstein operator 
$B_{n}:C\left[ a,b\right]
\rightarrow  U_n$
with non-decreasing nodes (resp. strictly increasing nodes), fixing $f_{0}$ and $f_{1}$, then
\[
\left(  \frac{f_{1}}{f_{0}}\right)  ^{\prime}\left(  x\right)  >0 \text{ for
all }x\in\left(  a,b\right) (\text{resp. for
all }x\in\left[ a,b\right]) .
\]
\end{corollary}

\begin{proof} We consider the non-decreasing case.
Recall that if $B_n$ exists, by hypothesis $f_0 > 0$ and  $\alpha_{n,k} > 0$,
so $\beta_{n,k} = f_0 (t_{n,k}) \ \alpha_{n,k} \ > 0$. 
Thus, we can use the implication a) $\implies$  c) 
from the preceding Theorem. Writing  
\[
\frac{d}{dx}\frac{f_{1}}{f_{0}}=\sum_{k=0}^{n-1}w_{k}q_{n-1,k},
\]
we have  $w_{k} \ge  0$ for $k = 0, \dots , n-1$.
Since $f_{1}/f_{0}$ is strictly increasing, it is non-constant, so 
$\frac{d}{dx}\frac{f_{1}}{f_{0}}\neq0$ and  thus some  coefficient $w_{j}$ is strictly
positive.   But now 
$$
\frac{d}{dx}\frac{f_{1}}{f_{0}} \ge w_j q_{n-1,j} > 0
$$  on $\left(  a,b\right)$.
\end{proof}

\begin{example} The condition $
\left(  \frac{f_{1}}{f_{0}}\right) ^{\prime}  >0$ on $\left[ a,b\right]$
 is not sufficient to ensure non-decreasing nodes, as \cite[Example 4.1]{AKR08b}
 shows: consider $\mathbb{P}_3[0,1]$ with the standard Bernstein basis, 
 let $f_0 = \mathbf{1}$, and let $f_1 (x) = 3 x /8 - x^2/2 + x^3/3$.
 Then $f_1^\prime (x) := (x - 1/2)^2 + 1/8
 	= 3 p_{2,0} (x) /8 -  p_{2,1} (x) /8 +  3 p_{2,2} (x) /8$, so by Theorem \ref{ThmBern2} 
	no generalized Bernstein operator $B_3 :  C([a,b], \mathbb{K}) \to \mathbb{P}_3[a,b]$, fixing $\mathbf{1}$ and $f_1$,
	can be defined via a non-decreasing sequence of nodes. 
	
	\end{example}

\begin{example} To finish, we revisit Example \ref{ex1}  under the
	light of the preceding results. A practical  advantage of applying,
	for instance, Corollary \ref{antimaz}, 
	is that one does not need to determine whether or not  $U_{n}$ 
	has a non-negative Bernstein basis. If it does not,
	a generalized Bernstein operator does not exist, and nothing
	else needs to be done. So it is enough to check that  $D_{f_{0}}U_{n}$ possesses a positive Bernstein basis, a task a priori simpler, since
	$D_{f_{0}}U_{n}$  has  one fewer dimension.
	
	 As before, we can find
	a positive  Bernstein basis of 
	$$
	D_{\mathbf{1}} \mathbb{E}_4 = 
	\operatorname{Span}\{ 1, x, x^2,  x^{5}\}
	$$ 
over $[a,b] =[-1,1]$, by writing  arbitrary linear combinations of the functions
$\{ 1, x, x^2, x^5\}$, and imposing the conditions of having precisely 3 zeros,
$k$  of them at $-1$ and the other
$3-k$ at $1$. Multiplying by $-1$ if needed, these basis functions can
be assumed to be non-negative at 0 (we shall see that they  are 
actually positive
inside $(-1, 1)$). In this way we find 
the following  Bernstein basis $\{p_{3,0}, \dots,p_{3,3}\}$
on $\left[ -1,1\right] $:
\begin{eqnarray*}
	p_{3,0}\left( x\right) &=& 1 - \frac{5}{2} x + \frac{5}{3}  x^2  - \frac{1}{6} 
	x^5, \\
	p_{3,1}\left( x\right) &=& 1 - \frac{1}{2} x -  x^2  + \frac{1}{2} 
	x^5, \\
	p_{3,2}\left( x\right) &=& 1 +  \frac{1}{2} x -   x^2  - \frac{1}{2} x^5, \\
	p_{3,3}\left( x\right) &=& 1 + \frac{5}{2} x + \frac{5}{3}  x^2  + \frac{1}{6} 
	x^5.
\end{eqnarray*}
These functions  have the 
correct number of
zeros at the endpoints. To see
that they form a positive basis, since $p_{3,k}(0) = 1 > 0$ for
$k=0,\dots, 3$, 
it suffices to show that they have no additional zeros inside
$(-1,1)$, which follows by factoring all the corresponding linear terms $(x +1)$ and
$(x - 1)$. In  each case we are left
with a second degree polynomial having no real roots. So  we are within the 
realm of
Theorem \ref{ThmBern2} or Corollary \ref{antimaz}. Since the derivative
of $f_1 (x) = x^3$ vanishes at 0, we conclude that 
	no generalized Bernstein operator $B_3 :  C([-1,1], \mathbb{K}) \to \mathbb{E}_4 [-1,1]$, fixing $\mathbf{1}$ and $f_1$,
can be defined via a non-decreasing sequence of nodes.

 Observe however 
that this result  is less informative than   Example \ref{ex1}, 
since it does not tell us whether a generalized Bernstein operator
can be defined, by dropping the requirement that nodes be non-decreasing.
\end{example}

\end{document}